\documentclass[a4paper,12pt]{amsart}
\usepackage{fullpage}
\usepackage[latin2]{inputenc}
\usepackage[english]{babel} 

\usepackage{tikz}
\usepackage{amsmath}
\usepackage{amsthm}
\usepackage{amssymb}
\usepackage{cite}
\usepackage{amsrefs}

\newtheorem{thm}{Theorem}[section]
\newtheorem{step}{Step}

\newtheorem{lem}[thm]{Lemma}
\newtheorem{defi}[thm]{Definition}

\newtheorem{notation}[thm]{Notation}

\newtheorem{claim}[thm]{Claim}
\newtheorem{observation}[thm]{Observation}
\newtheorem{corr}[thm]{Corollary}

\theoremstyle{remark}
\newtheorem{remark}[thm]{Remark}
\providecommand{\keywords}[1]{\textbf{\textit{Keywords ---}} #1}
\newcommand{\D}{\mathcal D}

\newcommand{\cH}{\mathcal H}
\newcommand{\C}{\mathcal C}

\begin{document}

\title{Independent sets in the union of two Hamiltonian cycles}

\makeatother
\author{Ron Aharoni}
\address{Department of Mathematics\\ Technion, Haifa, Israel}
\email[Ron Aharoni]{raharoni@gmail.com}
\thanks{\noindent The research of the first author was
supported by BSF grant no. $2006099$, by GIF grant no. I
$-879-124.6/2005$, by the Technion's research promotion fund, and by
the Discont Bank chair.}

\author{Daniel Soltész}
\address{Department of Computer Science and Information Theory \\ Budapest University of Technology and Economics}
\email[Daniel Soltész]{solteszd@math.bme.hu}
\thanks{\noindent The research of the second author was
supported by the Hungarian Foundation for Scientific Research Grant (OTKA) No. 108947}

\begin{abstract} Motivated by a question on the maximal number of vertex disjoint Schrijver graphs in the Kneser graph, we investigate the following function, denoted by $f(n,k)$: the maximal number of Hamiltonian cycles
on an $n$ element  set, such that no two cycles share a common
independent set of size more than $k$. We shall mainly be interested in the behavior of $f(n,k)$ when $k$ is a linear function of $n$, namely $k=cn$. We show a threshold phenomenon: there exists a constant $c_t$
such that for $c<c_t$, $f(n,cn)$ is bounded by a constant depending
only on $c$ and not on $n$, and for $c_t <c$, $f(n,cn)$ is
exponentially large in $n ~(n \to \infty)$. We prove  that $0.26 < c_t < 0.36$, but the exact value of $c_t$ is not determined. For the lower bound we prove a technical lemma, which for graphs that are the union of two Hamiltonian cycles establishes a relation between the independence number and the number of $K_4$ subgraphs. A corollary of this lemma is that if a graph $G$ on $n>12$ vertices is the union of  two Hamiltonian cycles and $\alpha(G)=n/4$, then $V(G)$
can be covered by vertex-disjoint $K_4$ subgraphs.


\end{abstract}


\maketitle

\section{Introduction}

In this paper we study a ``pigeonhole'' phenomenon for Hamiltonian cycles - 
in a large enough set of such cycles there are necessarily two that are close, in the sense that their union contains a large independent set (meaning that they are similar to each other). The motivation comes from 
Schrijver subgraphs of the Kneser graph. 
The {\em Kneser graph} $KG[n,k]$ has as vertices the $k$-subsets of $[n]$, two vertices being connected if the sets are disjoint. A celebrated result of Lov\'asz \cite{lovasz} is that the  chromatic number of $KG[n,k]$ is $n-2k+1$. His proof used topology, and it gave birth to the field of topological combinatorics. Later Schrijver proved that a relatively small induced subgraph of $KG[n,k]$ already has the same chromatic number. The vertices of this subgraph 
are those $k$-sets that are independent on a given, fixed, Hamiltonian cycle on $[n]$.
 The question we are interested in is what is the largest size of a set of vertex disjoint Schrijver subgraphs of $KG[n,k]$. Two Schrijver subgraphs are vertex disjoint if their Hamiltonian cycles do not share an independent set of size $k$, meaning that the union of their Hamiltonian cycles has independence number less than $k$. So, the question is on the maximal number of Hamiltonian cycles with a given bound on the independence number of each pairwise union.

Throughout the paper, unless otherwise stated the size of the vertex set of any graph mentioned is denoted by $n$. As usual,  $\alpha(G)$ denotes the maximal size of an independent set in a graph $G$. If $G$ and $H$ are graphs on the same ground set $V$, we write $G \cup H$ for the graph on $V$ with $E(G) \cup E(H)$ as edge set. A {\em Hamiltonian cycle} on $V$ is a simple cycle containing all vertices of $V$.

\begin{defi}
 $f(n,k)$ is the maximal size of a  set $\cH$ of Hamiltonian cycles on $n$ vertices, such that $\alpha(H_1 \cup H_2) \le k$ for every $H_1\neq H_2 \in \cH$.
\end{defi}

We study $f(n,k)$ in the case where $k$ is a linear function of $n$, namely $k=cn$. This is very natural as the independence number of a Hamiltonian cycle grows roughly like a linear function of $n$. Our main observation is the following threshold phenomenon.

\begin{thm}\label{threshold}
There is a constant $c_t$, such that for $c<c_t$ the function $f(n,cn)$ is bounded, and for $c_t<c$ the function $f(n,cn)$ is exponentially large in $n$.
\end{thm}

If $H$ is a Hamiltonian cycle, then $\alpha(H) =\lfloor \frac{n}{2}\rfloor$. Given two Hamiltonian cycles $H_1$ and $H_2$, their common independence number, $\alpha(H_1 \cup H_2)$, lies between $ \frac{n}{4}$ (this bound follows from Brooks' theorem) and $ \frac{n}{2}$. Thus, the trivial bounds for the threshold are  $0.25 \leq c_t \leq 0.5$. We improve these as follows.

\begin{thm} \label{bounds}
$$0.26627 \approx \frac{45}{169} \leq c_t \leq \frac{11}{30} \approx 0.3666. $$
\end{thm}

\begin{defi} A graph is said to be {\em two-miltonian} if it is the union of two Hamiltonian cycles.
\end{defi}

Besides the value of $c_t$, we are also interested in the first non-trivial values of the function $f$, namely  $f(n,n/2-1)$ and $f(n,n/4)$. We will show by an easy argument that $f(n,n/2-1) \sim 2^n$, and by a surprisingly hard one that $f(n,n/4) = 2$ except for $n=4,8$, where  $f(4,1)=f(8,2)=3$.

Since a two-miltonian graph satisfies $\Delta \leq 4$, the following results will be useful for us:

\begin{thm} \label{7/26}[Locke, Lou] \cite{7/26}
If $G$ is a connected $K_4$-free simple graph satisfying  $\Delta(G) \le 4$, then $\alpha(G) \geq (7n-4)/26 \approx 0.2692n.$ 
\end{thm}

Theorem~\ref{7/26} points towards the importance of $K_4$ subgraphs when $c$ is near $1/4$.

\begin{defi}
Given a two-miltonian graph $G$ we write $\zeta(G)$ for the number of copies of $K_4$s in $G$.
If $\zeta(G)=n/4$ (namely if the vertices of $G$ can be covered by  $K_4$s) then we say that $G$ is
{\em $K_4$-covered}.
\end{defi}

The most useful tool used in this paper is the following rather technical lemma.


\begin{lem} \label{technical}
Let $G$ be a two-miltonian graph on $n>13$ vertices. Let $G'$ be
obtained from  $G$ by removing all vertices in all copies of $K_4$.
Then there exists a graph $H$ with $V(H)=V(G')$ and $E(G') \subseteq
E(H)$, satisfying:
\begin{enumerate}
\item $H$ is connected. \label{connected}
\item $H$ is $K_4$-free.
\item $d_H(v) \leq d_G(v)$  for every vertex $v \in V(H)$,  with strict inequality at least for one vertex $v$ if $G$ is not $K_4$-free.
\item For every independent set $I$ of $H$ there exists a set $J$ consisting of a choice of one vertex from each $K_4$ in $G$, such that $I \cup J$
 is  independent in $G$. \label{strong}
\end{enumerate}
\end{lem}

Intuitively Lemma \ref{technical} states that if $G$ is two miltonian, we can use theorem \ref{7/26} on the $K_4$-free part of $G$ to obtain a large independent set and we can further enlarge it by adding a vertex from each $K_4$ maintaining independence. The authors feel that in Lemma \ref{technical} the assumption that $G$ is two-miltonian can be replaced by different assumptions, see Remark \ref{two-miltonian}. This Lemma is the core of the argument for the lower bound in Theorem \ref{bounds} and in the proof of the following theorem.

\begin{thm} \label{structural}
Let $G$ be a two-miltonian graph  on $n>12$ vertices. Then $\alpha(G)
=\frac{n}{4}$ if and only if $G$ is  $K_4$-covered.
\end{thm}

 Theorem \ref{structural} is sharp in the following sense: for $n=8,12$ there exist two-miltonian graphs with  $\alpha=n/4$ and $\zeta=n/4-1$. For general, not necessarily two-miltonian but  $\Delta(G) \leq 4$ graphs, the statement of Lemma \ref{technical}  and Theorem \ref{structural} are false. There exist non two-miltonian graphs on arbitrarily large ground sets with $\alpha=n/4$ and $\zeta \le n/8$, see Figure \ref{counterexample}.

\begin{figure}[htbp,scale=0.5]
\begin{center}
\begin{tikzpicture}

\begin{scope}[shift={(0,0)}]
\filldraw[black] (0,0) circle (2pt); 
\filldraw[black] (0,1) circle (2pt);
\filldraw[black] (1,0) circle (2pt);
\filldraw[black] (1,1) circle (2pt);
\draw (0,0) -- (0,1); 
\draw (0,0) -- (1,0);
\draw (0,0) -- (1,1);
\draw (0,1) -- (1,0);
\draw (0,1) -- (1,1);
\draw (1,0) -- (1,1);
\filldraw[black] (0,-1.5) circle (2pt);
\filldraw[black] (1,-1.5) circle (2pt);
\draw (0,0) to[out=200,in=150] (0,-1.5);
\draw (0,1) to[out=200,in=150] (0,-1.5);
\draw (1,0) to[out=-20,in=20] (1,-1.5);
\draw (1,1) to[out=-20,in=20] (1,-1.5);
\filldraw[black] (0,-2.5) circle (2pt); 
\filldraw[black] (1,-2.5) circle (2pt);
\draw (0,-1.5) -- (0,-2.5); 
\draw (0,-1.5) -- (1,-2.5);
\draw (1,-1.5) -- (0,-2.5);
\draw (1,-1.5) -- (1,-2.5);
\draw (1,-2.5) -- (0,-2.5);
\end{scope}

\begin{scope}[shift={(3,0)}]
\filldraw[black] (0,0) circle (2pt); 
\filldraw[black] (0,1) circle (2pt);
\filldraw[black] (1,0) circle (2pt);
\filldraw[black] (1,1) circle (2pt);
\draw (0,0) -- (0,1); 
\draw (0,0) -- (1,0);
\draw (0,0) -- (1,1);
\draw (0,1) -- (1,0);
\draw (0,1) -- (1,1);
\draw (1,0) -- (1,1);
\filldraw[black] (0,-1.5) circle (2pt);
\filldraw[black] (1,-1.5) circle (2pt);
\draw (0,0) to[out=200,in=150] (0,-1.5);
\draw (0,1) to[out=200,in=150] (0,-1.5);
\draw (1,0) to[out=-20,in=20] (1,-1.5);
\draw (1,1) to[out=-20,in=20] (1,-1.5);
\filldraw[black] (0,-2.5) circle (2pt); 
\filldraw[black] (1,-2.5) circle (2pt);
\draw (0,-1.5) -- (0,-2.5); 
\draw (0,-1.5) -- (1,-2.5);
\draw (1,-1.5) -- (0,-2.5);
\draw (1,-1.5) -- (1,-2.5);
\draw (1,-2.5) -- (0,-2.5);
\end{scope}

\begin{scope}[shift={(6,0)}]
\filldraw[black] (0,0) circle (2pt); 
\filldraw[black] (0,1) circle (2pt);
\filldraw[black] (1,0) circle (2pt);
\filldraw[black] (1,1) circle (2pt);
\draw (0,0) -- (0,1); 
\draw (0,0) -- (1,0);
\draw (0,0) -- (1,1);
\draw (0,1) -- (1,0);
\draw (0,1) -- (1,1);
\draw (1,0) -- (1,1);
\filldraw[black] (0,-1.5) circle (2pt);
\filldraw[black] (1,-1.5) circle (2pt);
\draw (0,0) to[out=200,in=150] (0,-1.5);
\draw (0,1) to[out=200,in=150] (0,-1.5);
\draw (1,0) to[out=-20,in=20] (1,-1.5);
\draw (1,1) to[out=-20,in=20] (1,-1.5);
\filldraw[black] (0,-2.5) circle (2pt); 
\filldraw[black] (1,-2.5) circle (2pt);
\draw (0,-1.5) -- (0,-2.5); 
\draw (0,-1.5) -- (1,-2.5);
\draw (1,-1.5) -- (0,-2.5);
\draw (1,-1.5) -- (1,-2.5);
\draw (1,-2.5) -- (0,-2.5);
\end{scope}

\draw (-1,-2.6) to[out=0,in=200] (0,-2.5);
\draw (1,-2.5) to[out=-20,in=200] (3,-2.5);
\draw (4,-2.5) to[out=-20,in=200] (6,-2.5);
\draw (7,-2.5) to[out=-20,in=180] (8,-2.6);

\end{tikzpicture}
\end{center}

\caption{The strip closes on itself. This is a connected graph that is not the union of two Hamiltonian cycles and it has independence number $n/4$ while only half of its vertices can be covered by $K_4$s. }
\label{counterexample}
\end{figure}
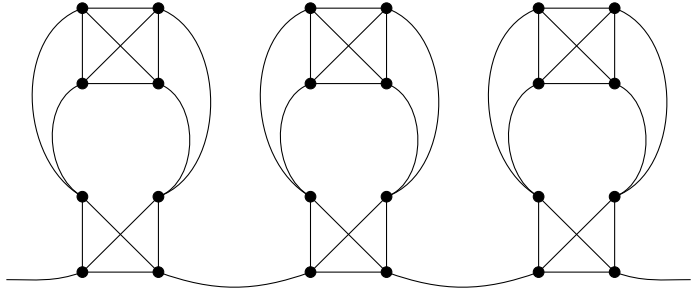

The paper is organized as follows. In Section \ref{sec:threshold} we prove the threshold phenomenon in the behavior of $f(n,cn)$, and using probabilistic arguments  we prove upper bounds on the threshold value $c_t$. In Section \ref{sec:k4free} we prove Lemma \ref{technical}. In  Section \ref{sec:nover4} we calculate $f(n,n/4)$ for all $n$.  In Section \ref{sec:lowerbounds} we prove lower bounds on $c_t$.



\section{A threshold phenomenon}\label{sec:threshold}



In this section we  prove Theorem \ref{threshold}.
The core of the proof is the following lemma:

\begin{lem} \label{semi-random}
Let $\varepsilon>0$ and $n_0,c_0,k_0$ be constants. If $f(n_0, c_0n_0) \ge k_0$ then the function
$f\left(n,\left(\frac{1}{k_0}\frac{1}{2}+\frac{k_0-1}{k_0}c_0+\frac{1}{2n_0}
+\varepsilon \right)n \right)$ grows exponentially in $n$.
\end{lem}

The proof will use a standard concentration result:

\begin{lem}\label{average}
If the elements of two sequences $\sigma, \tau$ of length $N$ are chosen at random from a set of size $k$ then
$$Pr \left( |\{a: \sigma(a) = \tau(a)\}| > \frac{N}{k}+\varepsilon \right) < \exp(-2 \varepsilon N).$$

\end{lem}
\begin{proof} : For given $a$, $Pr(\sigma(a) = \tau(a)) =\frac{1}{k}$, and hence the
expected number of indices in which $\sigma$ and $\tau$ have identical elements is $\frac{N}{k}$. The result now follows by the Chernoff inequality (see, e.g.,    \cite{alonspencer}).
\end{proof}

\begin{proof}
[Proof of Lemma \ref{semi-random}]

Let $S_1, \ldots ,S_N$ be disjoint copies of a set of size $n_0$,
where $N$ is an even number to be specified below. Let $V=\bigcup_{i
\le N}S_i$, and write  $n=|V|=Nn_0$. An {\em $N$-chain} is an
$N$-tuple of cycles  ${\D}=(C_{i_1}, \ldots ,C_{i_N})$, where
$C_{i_a} \in \C$ is a Hamiltonian cycle chosen from $\C$ on $S_a$.

 Let $m=\exp( \varepsilon N)= \exp(\varepsilon n/n_0)$. Choose  $m$ $N$-chains $\D^1, \ldots ,\D^m$,
  forming each $\D^h, ~h \le m$ by choosing a cycle
$C^h_i~~(i\le N)$ in each $\D^h$ at random from $\C$,
 uniformly  and independently. By Lemma \ref{average}  the
probability that  there
exists a pair $\D^j, \D^h$ for which $|\{a \mid C^h_{i_a}=C^j_{i_a}\}| > \frac{N}{k_0}+\varepsilon $ is smaller than
$\binom{m}{2} \exp(-2 \varepsilon N)$, which is less than $1$. Thus for every $N$ there exist $\exp( \varepsilon N)$ $N$-chains $\D^j$, such that
 $|\{a \mid C^h_{i_a}=C^j_{i_a}\}| \le \frac{N}{k_0}+\varepsilon $  whenever $j \neq h$.
Writing $F_j=\bigcup \D_j$, we then have, for every pair $j, h \le m$:

$$\alpha(F_j \cup F_h) \le (N/k_0+\varepsilon N)n_0/2 + (N(k_0-1)/k_0 - \varepsilon N)c_0n_0. $$

Here the first term comes from the cycles for indices $a$ for which $C^j_{i_a}=C^h_{i_a}$. The second term comes from the other cycles, applying the assumption of the theorem, that    $\alpha(C_a \cup C_b) \le c_0n_0$ whenever $a<b\le k_0$.

The next step is to turn each $F_j$ into a Hamiltonian cycle. Pick a vertex $v_a$ in each copy $S_a$ of $S$,
  and for each $j$  delete an edge of $F_j$ incident with $v_a$. This changes $F_j$  into the union of $N$ paths,
  each having a vertex $v_a$ as one of its endpoints.
Put a matching arbitrarily on the  vertices $v_a$ (this is where we are using the fact that $N$ is even), thus making $F_j$ to be the union $F'_j$ of $N/2$ disjoint paths. Now form a Hamiltonian cycle $B_j$  by adding $N/2$ new edges, chosen arbitrarily, to  $F'_j$.

Since the vertices $v_a$ are connected by a matching, for every pair $(j,h)$ of indices an independent set in $B_j \cup B_h$
 contains at most $\frac{N}{2}$ vertices $v_a$, and hence

 $$\alpha(B_j \cup B_h)\le \sum_{a \le N}\alpha(C^j_{i_a} \cup C^h_{i_a})+\frac{N}{2}\le $$

$$(N/k_0+\varepsilon N)n_0/2 + (N(k_0-1)/k_0 - \varepsilon N)c_0n_0 + N/2 $$
yielding the independence ratio
$$\frac{(N/k_0+\varepsilon N)n_0/2 + (N(k_0-1)/k_0 - \varepsilon N)c_0n_0 + N/2}{n}= $$
$$\left(\frac{1}{k_0}+\varepsilon \right) \frac{1}{2}+\left(\frac{k_0-1}{k_0}-\varepsilon \right)c_0+\frac{1}{2n_0}\leq \frac{1}{k_0}\frac{1}{2}+\frac{k_0-1}{k_0}c_0+\frac{1}{2n_0} +\varepsilon. $$

This proves the existence of exponentially large systems of Hamiltonian cycles with the appropriate size of independent sets in each union of two Hamiltonian cycles, for ground sets divisible by $2n_0$. The lemma for  ground sets of general size follows directly.
\end{proof}

To deduce Theorem \ref{threshold} from Lemma~\ref{semi-random}, let us first re-formulate the theorem to an equivalent form:

\begin{thm} \label{sharp}(re-formulated)~~
If $\limsup_{n \rightarrow \infty} f(n,c_0n) = \infty$  then
 for every $\varepsilon >0$ there exists  $\gamma=\gamma(\varepsilon) >1$ such that for large enough $n$ we have:
$$ f(n,(c_0+\varepsilon)n) > \gamma^n. $$
\end{thm}

\begin{proof}  Let  $k_0 \geq \frac{3}{2\varepsilon}$ and $\varepsilon=\frac{\varepsilon}{3}$. By the assumption there exists
 $n_0 \geq \frac{3}{2\varepsilon}$ for which $f(n_0,c_0n_0) \geq k_0$. For large enough $k_0$ we have  $ \frac{1}{k_0}\frac{1}{2}+\frac{k_0-1}{k_0}c_0+\frac{1}{2n_0} +\varepsilon \leq c_0+\varepsilon $, and thus the
 theorem follows by Lemma \ref{semi-random}.
\end{proof}

Lemma \ref{semi-random} can be used to yield not only the existence of the threshold $c_t$, but also an upper bound. We prove the upper bound in theorem \ref{bounds}.

\begin{claim} \label{bestupper}
$c_t \leq 11/30 \approx 0.3666$
\end{claim}
\begin{proof}
Let $n$ be odd and divisible by $3$. Take as ground set the elements of $\mathbb{Z}_n$ (residue classes modulo $n$). We define the edge sets of two cycles and three forests on $n$ vertices as follows.
$$E(C_1) := \{ (k,k+1) | k \in \mathbb{Z}_n\} \quad E(C_2):= \{(k,k+2) | k \in \mathbb{Z}_n \} $$
$$E(C'_{3}) := \{(3k,3k+2),(3k,3k+4) | k \in \mathbb{Z}_n\} $$
$$E(C'_{4}) := \{(3k+1,3k+3),(3k+1,3k+5) | k \in \mathbb{Z}_n\} $$
$$E(C'_{5}) := \{(3k+2,3k+4),(3k+2,3k+6) | k \in \mathbb{Z}_n\} $$
Connect the connected components of $C'_{3},C'_{4},C'_{5}$ to form Hamiltonian cycles $C_{3},C_{4},C_{5}$ arbitrarily. It is easy to verify that for $0 \leq i < j \leq 5$ the graph $C_i \cup C_j$ can be covered by vertex disjoint triangles, thus it has independence number at most $n/3$. Now we can use Lemma \ref{semi-random} with $k_0=5$, $c_0 = 1/3$ and $n_0$ odd and divisible by three, thus we get that for every $\varepsilon$

$$f\left(n,\left(\frac{1}{k_0}\frac{1}{2}+\frac{k_0-1}{k_0}c_0+\frac{1}{2n_0} +\varepsilon \right)n \right)=f\left(n,\left(\frac{11}{30}+\frac{1}{2n_0} +\varepsilon \right)n \right)$$

is exponentially large in $n$. Since we can choose $n_0$ to be arbitrarily large, we conclude that $c_t \leq 11/30$.
\end{proof}

\section{$K_4$-free graphs}\label{sec:k4free}

A  tool we shall use in two contexts is:

\begin{thm}\label{notevenmyfinalform}[Locke, Lou] \cite{7/26}
Let $G$ be as in the above theorem, and write $e=|E(G)|$. Then

$$e-9n+26\alpha(G) \geq -4.  $$
\end{thm}

\begin{remark}  Theorem \ref{7/26} follows from Theorem \ref{notevenmyfinalform} and the observation that $\Delta(G) \le 4$ implies $e \le 2n$.
 Theorem \ref{notevenmyfinalform} is best possible in the sense that there are infinitely many graphs for which equality is attained. For a characterisation of these graphs, and a slight improvement on the constant $-4$ for other graphs, see \cite{7/26}. By contrast,
it is not known whether Theorem \ref{7/26} is best possible for large $n$.
\end{remark}

Now we prove Lemma \ref{technical}.

\begin{proof}

In the proof below, $G$ will always denote a two-miltonian graph.

\begin{defi}
A connected $K_4$-coverable induced subgraph  of $G$ is called an
{\em archipelago}. An archipelago  is said to be {\em cyclic} if
contains an induced cycle of length at least $4$, and otherwise it
is called {\em acyclic}. The set of edges in an archipelago $K$ that
do not lie in a $K_4$ is denoted by $M(K)$.
\end{defi}

Since $G$ is two-miltonian $\Delta(G)\le 4$,   implying that the
$K_4$s in $K$ are vertex disjoint, and that $M(K)$ is a matching, consisting of edges connecting $K_4$s.
In an acyclic  archipelago the $K_4$s are connected in a tree-like
fashion.

\begin{notation}
The {\em neighborhood} $N(S)$ of a set $S$ of vertices is the set of vertices connected to $S$ and not belonging to $S$ itself.
\end{notation}

In other words, $N(S)$ is the {\em open} version of ``neighborhood''. Since every $K_4$ in $G$ sends out at least $4$ edges, we have:

\begin{claim}\label{4edgesout}
An acyclic archipelago sends at least
four edges to its neighborhood.
\end{claim}

We shall  remove the $K_4$s from $G$ one archipelago at a time.
 The next observation and claim explain why if the archipelago is cyclic we can plainly remove it, without having to worry about
 (\ref{strong}).

\begin{observation} \label{greedy}
Let $J$ be a connected graph with  $\Delta(J) \le 4$,  and let $I$
be a non-empty independent set in $J$. Then there is an independent
set $I'$ of $J$ containing $I$,  of size at least $|I|+ |V(J)
\setminus N[I]|/4$.
\end{observation}

 \begin{proof}
If $N(I) =V(J)$, then  taking $I'=I$ does the job. Otherwise,  by
the assumption of  connectivity, there exists $v_1 \in V(J)
\setminus N(I)$ connected to $N(I)$ by an edge. Let $I_1=I\cup
\{v_1\}$. By the assumption that $\Delta(J) \le 4$ and by the fact
that $v_1$ is connected to $N(I)$, we have $|N(I_1)|\le |N(I)|+4$.
If $N(I_1)=V(J)$ then we can take $I'=I_1$. Otherwise we add to
$I_1$ a vertex $v_2 \in V(J) \setminus N(I_1)$ that is connected to
$N(I_1)$, and continue.

 \end{proof}

\begin{claim} \label{k4cycles}
If $K$ is a cyclic  archipelago, then there exists an independent set $I \subseteq V(K)$ of size $|V(K)|/4$ (namely, $I$ contains
one vertex from each $K_4$) such that $N(I) \subseteq V(K)$.
\end{claim}
\begin{proof}
Let $M=M(K)$. Since $K$ is cyclic,  there exists in $K$ an induced cycle $C$ of length at least $4$. The edges of $C$ alternate between $M$ and
$E(K) \setminus M$, and hence $C$ is even. The set $I_0$ consisting of the odd vertices in $C$ is then  independent, and $N(I_0) \subseteq V(K)$.

 Let $I$ be the independent set obtained by using Observation \ref{greedy} starting from $I_0$
 in the graph induced by the vertices of $K$. Thus $|I| \ge |V(K)|/4$, and since $K$ is  $K_4$-coverable, in fact $|I|=|V(K)|/4$.
Observe that every vertex added to $I$ in the algorithm of Observation \ref{greedy}  has a neighbor among
the previous vertices, which belong to $V(K)$, and three  neighbors in its own $K_4$. Thus the newly vertex cannot have a neighbor outside $K$.
\end{proof}

 The  same argument   yields:

\begin{claim} \label{tooeasy}
If $K$ is an archipelago then for each vertex $v \in V(K)$ having a neighbor in $V(G) \setminus V(K)$ there is an independent set $I \subseteq V(K)$ containing $v$,  such that $|I|=|V(K)|/4$ and no vertex in $I \setminus \{v\}$  has a neighbor in $V(G) \setminus V(K)$.
\end{claim}


Claim \ref{tooeasy} is the main tool we shall use in the proof of the lemma, allowing us to take care of acyclic archipelagos $K$ that have non-independent neighborhoods. For such $K$ every independent set $J$ in our future $H$  omits a vertex in its neighborhood, and thus by Claim \ref{tooeasy} there exists an independent set $I_K \subseteq V(K)$ of size $|V(K)|/4$, such that $I_K \cup J$ is independent. Thus Claims \ref{k4cycles} and \ref{tooeasy}  allow us to remove with no penalty all cyclic archipelagos and all archipelagos with non-independent neighborhoods, towards the  removal of all $K_4$s.
Thus, the problem is posed by acyclic archipelagos with independent  neighborhood. Our strategy in this case is to add an edge inside this neighborhood, taking care not to generate a new $K_4$.

\begin{remark}\label{constantneighborhood}
The set $N(K)$ of vertices in the neighborhood of an archipelago $K$ remains the same throughout the process, since the edges we add are inside the neighborhood of a deleted archipelago, and as such they do not belong to $K$. Edges may be added inside $N(K)$, but as remarked this is in our favor.
\end{remark}

We shall remove the archipelagos in a special order, aimed to preserve  useful properties of $G$.

\begin{claim}
If $K$ is an acyclic archipelago with an independent neighborhood of size $2$ and $\zeta(K)>1$ then $V(K)\cup N(K)=V(G)$.
\end{claim}

\begin{proof}
Since $K$ is acyclic, $|M(K)| = \zeta(K)-1$. But there are exactly four edges leaving every $K_4$. Thus there are $4 \zeta(K)-2(\zeta(K)-1)=2\zeta(K)+2 \ge 6$ edges between $K$ and $N(K)$. Since $|N(K)|\le 2$  there is a vertex $v \in N(K)$ that receives $3$ edges from $K$, two of them being from the same Hamiltonian cycle $H_1$. Assume for contradiction that $V(K) \cup N(K)\neq V(G)$. Then deleting the other vertex of $N(K)$, if such exists, disconnects  $V(K)\cup \{v\}$ from the rest of
 $V(G)$ in $H_1$
(remembering that there are no other vertices in $N(K)$). This contradicts the fact that $H_1$ is $2$-connected.
\end{proof}

By the claim, we may assume that every acyclic archipelago with an independent neighborhood of size two has only one $K_4$.  We call such an archipelago {\em small} (see  Figure \ref{first case1}).

\begin{figure}[htbp,scale=0.5]
\begin{center}
\begin{tikzpicture}
\filldraw[black] (0,0) circle (2pt); 
\filldraw[black] (0,1) circle (2pt);
\filldraw[black] (1,0) circle (2pt);
\filldraw[black] (1,1) circle (2pt);
\draw (0,0) -- (0,1); 
\draw (0,0) -- (1,0);
\draw (0,0) -- (1,1);
\draw (0,1) -- (1,0);
\draw (0,1) -- (1,1);
\draw (1,0) -- (1,1);
\filldraw[black] (0,-1.5) circle (2pt)node[anchor=south] {$A$};
\filldraw[black] (1,-1.5) circle (2pt)node[anchor=south] {$B$};
\draw (0,0) to[out=200,in=150] (0,-1.5);
\draw (0,1) to[out=200,in=150] (0,-1.5);
\draw (1,0) to[out=-20,in=20] (1,-1.5);
\draw (1,1) to[out=-20,in=20] (1,-1.5);
\end{tikzpicture}
\caption{A small archipelago.}
\label{first case1}
\end{center}
\end{figure}
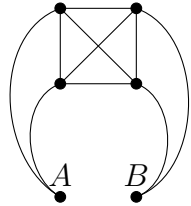

\begin{step}\label{smallnbd}
Removing small archipelagos.
\end{step}\label{removingn2}

We delete all small  archipelagos one by one,  connecting their two neighbors at each step. Each such deletion+connecting is called below an {\em operation}.

\begin{claim}
No new $K_4$s  are formed by this step.
\end{claim}
\begin{proof}
Consider first the first operation. It could result in a new $K_4$
only if $G$ contains the graph in Figure \ref{first case2} as a
subgraph. Since $n>13$, each of the two Hamiltonian cycles whose
union is $G$ must reach  this subgraph from the rest of the graph
via $C$ or $D$, and leave it from the other vertex in this pair.
Thus neither cycle can contain the edge $CD$, a contradiction.

\begin{figure}[htbp,scale=0.5]
\begin{center}
\begin{tikzpicture}
\filldraw[black] (0,0) circle (2pt); 
\filldraw[black] (0,1) circle (2pt);
\filldraw[black] (1,0) circle (2pt);
\filldraw[black] (1,1) circle (2pt);
\draw (0,0) -- (0,1); 
\draw (0,0) -- (1,0);
\draw (0,0) -- (1,1);
\draw (0,1) -- (1,0);
\draw (0,1) -- (1,1);
\draw (1,0) -- (1,1);
\filldraw[black] (0,-1.5) circle (2pt)node[anchor=south] {$A$};
\filldraw[black] (1,-1.5) circle (2pt)node[anchor=south] {$B$};
\draw (0,0) to[out=200,in=150] (0,-1.5);
\draw (0,1) to[out=200,in=150] (0,-1.5);
\draw (1,0) to[out=-20,in=20] (1,-1.5);
\draw (1,1) to[out=-20,in=20] (1,-1.5);
\filldraw[black] (0,-2.5) circle (2pt)node[anchor=north] {$C$}; 
\filldraw[black] (1,-2.5) circle (2pt)node[anchor=north] {$D$};
\draw (0,-1.5) -- (0,-2.5); 
\draw (0,-1.5) -- (1,-2.5);
\draw (1,-1.5) -- (0,-2.5);
\draw (1,-1.5) -- (1,-2.5);
\draw (1,-2.5) -- (0,-2.5);
\end{tikzpicture}
\caption{Impossible at the first replacement.}
\label{first case2}
\end{center}
\end{figure}
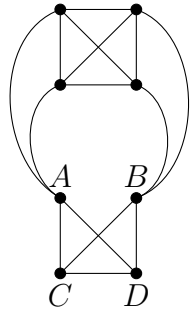

 Suppose, for contradiction, that in the chain of operations a new $K_4$ is generated.
As before, the graph obtained so far necessarily is as in Figure \ref{first case2}. Since  such a subgraph cannot be present in $G$, some edges have to come from  previous operations on small archipelagos. Observe that each such operaton reduces the degrees of the vertices of the newly added edge. Thus the only edge in Figure \ref{first case2} that could come from a previous deletion is the edge connecting $C$ to $D$. In this case, before that deletion our graph had to look like in figure \ref{first case3}.

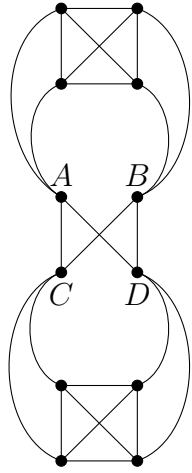
\begin{figure}[htbp,scale=0.5]
\begin{center}
\begin{tikzpicture}
\filldraw[black] (0,0) circle (2pt); 
\filldraw[black] (0,1) circle (2pt);
\filldraw[black] (1,0) circle (2pt);
\filldraw[black] (1,1) circle (2pt);
\draw (0,0) -- (0,1); 
\draw (0,0) -- (1,0);
\draw (0,0) -- (1,1);
\draw (0,1) -- (1,0);
\draw (0,1) -- (1,1);
\draw (1,0) -- (1,1);
\filldraw[black] (0,-1.5) circle (2pt)node[anchor=south] {$A$};
\filldraw[black] (1,-1.5) circle (2pt)node[anchor=south] {$B$};
\draw (0,0) to[out=200,in=150] (0,-1.5);
\draw (0,1) to[out=200,in=150] (0,-1.5);
\draw (1,0) to[out=-20,in=20] (1,-1.5);
\draw (1,1) to[out=-20,in=20] (1,-1.5);
\filldraw[black] (0,-2.5) circle (2pt)node[anchor=north] {$C$}; 
\filldraw[black] (1,-2.5) circle (2pt)node[anchor=north] {$D$};
\draw (0,-1.5) -- (0,-2.5); 
\draw (0,-1.5) -- (1,-2.5);
\draw (1,-1.5) -- (0,-2.5);
\draw (1,-1.5) -- (1,-2.5);
\filldraw[black] (0,-4) circle (2pt); 
\filldraw[black] (1,-4) circle (2pt);
\filldraw[black] (0,-5) circle (2pt);
\filldraw[black] (1,-5) circle (2pt);
\draw (0,-4) -- (1,-4);
\draw (0,-4) -- (0,-5);
\draw (0,-4) -- (1,-5);
\draw (1,-4) -- (0,-5);
\draw (1,-4) -- (1,-5);
\draw (0,-5) -- (1,-5);
\draw (0,-4) to[out=160,in=200] (0,-2.5);
\draw (0,-5) to[out=160,in=200] (0,-2.5);
\draw (1,-4) to[out=20,in=-20] (1,-2.5);
\draw (1,-5) to[out=20,in=-20] (1,-2.5);
\end{tikzpicture}
\caption{Impossible, unless $n=12$.}
\label{first case3}
\end{center}
\end{figure}

But in this case every vertex has degree four, thus $G$ consists of just these  $12$ vertices, contradicting our assumption that $n>13$.
\end{proof}

\begin{claim}\label{twmpreserved}
Each operation results in a two-miltonian graph.
\end{claim}
\begin{proof}
By induction on the number of operations. Assuming that after a deletion the resulting graph is the union of two Hamiltonian cycles $H_1$ and $H_2$, replace in each of $H_1, H_2$ the detour through the archipelago by the newly added edge.
\end{proof}


This concludes Step \ref{smallnbd}. The resulting graph $G_1$   is  two-miltonian by Claim
\ref{twmpreserved}, and it contains no acyclic archipelagos with independent
neighborhoods of size two.   Note also that $G_1$ is
a supergraph of $G'$, and hence it is enough to prove Lemma \ref{technical} with $G_1$ replacing $G$.

\begin{step} \label{conn}
Taking care of connectedness.
\end{step}
In this step we add edges,
so as to make the graph $G_1$ connected. These edges will remain in the next steps, and so we shall not have to worry about connectedness  from this point on.

Let $H_1$ be one of the two Hamiltonian cycles forming $G_1$. Let $G_1'$ be the graph obtained from $G_1$ by removing
all $K_4$s from it, and let
 $C_1, \ldots, C_m$ be the
connected components of $G'_1$. Define an auxiliary graph $A$ on the vertex set $V(A):= \{C_1, \ldots, C_m\}$,
 two  vertices $C_i$ and $C_j$ being connected if there is a path contained in $H_1$,
 whose one endpoint is in $C_i$ and the other in $C_j$ and all the other vertices lie in a single archipelago.
 Since $H_1$ is  Hamiltonian, the graph $A$ is connected. Choose a spanning tree of $A$, and let
 $P_1, \ldots , P_{m-1}$ be subpaths of $H_1$ associated with each edge of the spanning tree.  For each path $P_i$ going through an archipelago $K_i$ connect the  endpoints of $P_i$ in the graph $G_1$, and delete all vertices of the archipelago $K_i$. The graph $G_2$ obtained this way is connected, it does not contain any new $K_4$s and $\Delta(G_2) \le 4$.

 \begin{remark}\label{decreasingdegree}
 If $G_1'$ was not connected then the degree of
some vertices  decreases. This is true since
removing an archipelago reduces the total degree of the vertices adjacent to it by at least $4$, and the addition of an edge increases it only by $2$.
\end{remark}

The construction also yields:
\begin{remark}\label{non2con}
Edges added in this process do not belong to a cycle in $G_2$
\end{remark}




\begin{claim} \label{happy}
Let $K$  be an acyclic archipelago in a graph of maximum degree at most four. Suppose that the neighborhood $N(K)$ is independent and has size $4$ or more.
Then we can delete $K$  and connect two vertices in $N(K)$ by an edge, so that the new edge doesn't generate a new $K_4$.
\end{claim}
\begin{proof}
  Assuming negation, for every pair $p=\{x, y\}$ of vertices in $N(K) \setminus V(K)$ there exists a pair   $q(p) = \{u,v\} \subseteq V(G)\setminus N(K)$ such that  all pairs among $x,y,u,v$ apart from $xy$ are edges in $G$. We say that $q(p)$ is a {\em complementary pair} of $p$.

 Suppose first that there exist two  pairs $p_1,p_2$ such that $q(p_1) \neq q(p_2)$.
 If $p_1 \cap p_2 =\emptyset$ let $p_3$ be a pair meeting both. Then  $q(p_3) \neq q(p_i)$ for $i=1$ or $i=2$ (or both), proving that there exist non-disjoint pairs $p,p'$ with $q(p)\neq q(p')$.
 Let $v^*$ be the vertex in $p \cap p'$.
 If  $q(p) \cap q(p')=\emptyset$ , then $v^*$ is connected to the four vertices in $q(p) \cup q(p')$, and since as a member of $I$ it is also connected to a vertex in $K$, its degree in $G$ is at least $5$, a contradiction.
  On the other hand, if there exists a vertex $v^{**} \in q(p) \cap q(p')$,
  then  $v^{**}$ is connected to the five vertices in $q(p) \cup q(p') \cup p \cup p' \setminus \{v^{**}\}$, again a contradiction. \end{proof}

\begin{remark} \label{weakened}
 The proof yields a stronger result:   it suffices to assume that  in the independent neighborhood of the acyclic archipelago  not all pairs have the same complementary pair.
\end{remark}

\begin{step} \label{hard}
Deleting   acyclic archipelagos with independent neighborhoods of size  $3$.
\end{step}
Let $K$ be an  acyclic archipelago such that $N(K)$ is  independent and has size
$3$, say $N(K)=\{A,B,C\}$. By Remark \ref{weakened} if no pair in $N(K)$ can be connected without generating a $K_4$,
 all pairs must have
the same complementary pair, and thus $A,B,C$ are all connected to two vertices, $O_1$ and $O_2$.
In such a case we call $K$ {\em forbidden} and the $5$-vertex subgraph showing this {\em forbidding for $K$}.

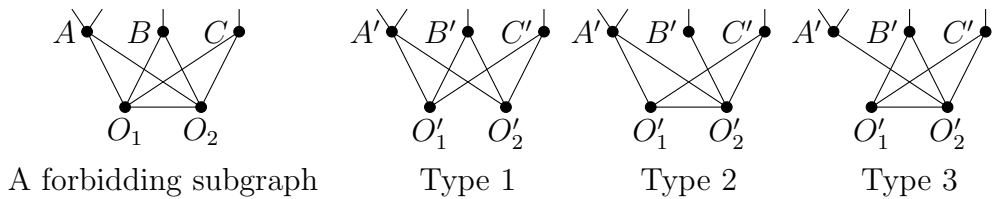
\begin{figure}[htbp,scale=0.5]
\begin{center}
\begin{tikzpicture}
\filldraw[black] (0,0) circle (2pt)node[anchor=east] {$A$}; 
\filldraw[black] (1,0) circle (2pt)node[anchor=east] {$B$};
\filldraw[black] (2,0) circle (2pt)node[anchor=east] {$C$};
\filldraw[black] (0.5,-1) circle (2pt)node[anchor=north] {$O_1$};
\filldraw[black] (1.5,-1) circle (2pt)node[anchor=north] {$O_2$};
\draw (0,0) -- (-0.2,0.3); 
\draw (0,0) -- (0.2,0.3); 
\draw (1,0) -- (1,0.3); 
\draw (2,0) -- (2,0.3); 
\draw (0,0) -- (0.5,-1); 
\draw (0,0) -- (1.5,-1);
\draw (1,0) -- (0.5,-1);
\draw (1,0) -- (1.5,-1);
\draw (2,0) -- (0.5,-1);
\draw (2,0) -- (1.5,-1);
\draw (0.5,-1) -- (1.5,-1);
\node[](The forbidden graph.) at (1,-2){A forbidding subgraph};
\end{tikzpicture}
\begin{tikzpicture}
\filldraw[black] (0,0) circle (2pt)node[anchor=east] {$A'$}; 
\filldraw[black] (1,0) circle (2pt)node[anchor=east] {$B'$};
\filldraw[black] (2,0) circle (2pt)node[anchor=east] {$C'$};
\filldraw[black] (0.5,-1) circle (2pt)node[anchor=north] {$O'_1$};
\filldraw[black] (1.5,-1) circle (2pt)node[anchor=north] {$O'_2$};
\draw (0,0) -- (-0.2,0.3); 
\draw (0,0) -- (0.2,0.3); 
\draw (1,0) -- (1,0.3); 
\draw (2,0) -- (2,0.3); 
\draw (0,0) -- (0.5,-1); 
\draw (0,0) -- (1.5,-1);
\draw (1,0) -- (0.5,-1);
\draw (1,0) -- (1.5,-1);
\draw (2,0) -- (0.5,-1);
\draw (2,0) -- (1.5,-1);
\node[](Type $1$) at (1,-2){Type $1$};
\end{tikzpicture}
\begin{tikzpicture}
\filldraw[black] (0,0) circle (2pt)node[anchor=east] {$A'$}; 
\filldraw[black] (1,0) circle (2pt)node[anchor=east] {$B'$};
\filldraw[black] (2,0) circle (2pt)node[anchor=east] {$C'$};
\filldraw[black] (0.5,-1) circle (2pt)node[anchor=north] {$O'_1$};
\filldraw[black] (1.5,-1) circle (2pt)node[anchor=north] {$O'_2$};
\draw (0,0) -- (-0.2,0.3); 
\draw (0,0) -- (0.2,0.3); 
\draw (1,0) -- (1,0.3); 
\draw (2,0) -- (2,0.3); 
\draw (0,0) -- (0.5,-1); 
\draw (0,0) -- (1.5,-1);
\draw (1,0) -- (1.5,-1);
\draw (2,0) -- (0.5,-1);
\draw (2,0) -- (1.5,-1);
\draw (0.5,-1) -- (1.5,-1);
\node[](Type $2$) at (1,-2){Type $2$};
\end{tikzpicture}
\begin{tikzpicture}
\filldraw[black] (0,0) circle (2pt)node[anchor=east] {$A'$}; 
\filldraw[black] (1,0) circle (2pt)node[anchor=east] {$B'$};
\filldraw[black] (2,0) circle (2pt)node[anchor=east] {$C'$};
\filldraw[black] (0.5,-1) circle (2pt)node[anchor=north] {$O'_1$};
\filldraw[black] (1.5,-1) circle (2pt)node[anchor=north] {$O'_2$};
\draw (0,0) -- (-0.2,0.3); 
\draw (0,0) -- (0.2,0.3); 
\draw (1,0) -- (1,0.3); 
\draw (2,0) -- (2,0.3); 
\draw (0,0) -- (1.5,-1);
\draw (1,0) -- (0.5,-1);
\draw (1,0) -- (1.5,-1);
\draw (2,0) -- (0.5,-1);
\draw (2,0) -- (1.5,-1);
\draw (0.5,-1) -- (1.5,-1);
\node[](Type $3$) at (1,-2){Type $3$};
\end{tikzpicture}
\caption{The forbidden subgraph $K$ is invisible in the picture. The Type $1,2,3$ \textit{risking} subgraphs are the three graphs (up to isomorphism) that are one edge short to be forbidding.}
\label{disgusting}
\end{center}
\end{figure}

\begin{claim}
$G_2$ does not contain a forbidden archipelago.
\end{claim}
\begin{proof}
Suppose that $G_2$ contains a forbidden archipelago with forbidding subgraph $F$, with vertices denoted as in the figure.
 Since $A,B,C \in N(K)$, they do not belong to any archipelago $J$, or else $K$ and $J$, being connected, would be contained in the same archipelago. The same is true for $O_1,O_2$, since they are of degree $4$ in $F$.
By Remark \ref{non2con} no edges in $F$
were added in Step \ref{conn}. Also, no edges inside $F$
was added in Step \ref{removingn2}, since the endpoints of any edge that
is added in Step \ref{removingn2} have maximum degree three. Thus $F$ is also a subgraph of $G$.
The archipelago $K$ consists of at most two $K_4$s, since
otherwise the degrees of some of the $A,B,C$ would be larger than
four.
But then the subgraph consisting of the union of $K$
 and the vertices $A,B,C,O_1,O_2$ has at most $13$
vertices, and it is connected to the rest of the graph by at most
two edges. Since the endpoints of these two edges have degree $4$,
this contradicts the fact that $G$ is  two-miltonian on more than
$13$ vertices.
\end{proof}

 We next delete archipelagos one by one, taking care not to generate a forbidden archipelago. The risk is that connecting two neighbors of a deleted archipelago may create a forbidding subgraph for some other archipelago.
Up to isomorphism, there are  three subgraphs  that are one edge short of the forbidding subgraph, see the graphs named
Type $1,2,3$ in Figure \ref{disgusting} (a step that helps in realizing this is noting that  in the forbidding subgraph the role of $O_1,O_2$ and $B,C$ is
symmetric). We will say that an  acyclic
archipelago $K'$ is {\em risky} if it has an independent neighborhood consisting of three vertices that are connected, besides to vertices in $K'$, to vertices $O_1, O_2$, forming a graph $Y$ of one of these three types. We say that the subgraph $Y$ is {\em risking} for $K$.
By Claim \ref{4edgesout} $K$ sends at least
four edges to its neighborhood, and hence at least one of the vertices $A,B,C$ receives from $K$  two
edges. As in the figure, we denote this vertex by $A$, and whenever ``$A$'' is used in this context we assume that it has degree at least $2$ to the risky archipelago.

\begin{claim} \label{worst part}
Let $K'$ be a risky  archipelago  contained in $G_2$,
and denote the vertices in its risking subgraph as in Figure \ref{disgusting}.
Then
deleting $K'$ and
connecting $A'$ to $B'$ does not generate  a
forbidding subgraph for some other  archipelago.
\end{claim}
\begin{proof}

 Let $Y$ be the risking subgraph of $K'$. Suppose, by negation,  that
 deleting $K'$ and adding
 the edge $e=A'B'$
 generates
 a forbidden  archipelago $K$. This was born from a risking graph $Z$
 for $K$.
Denote the vertices of  $Z$ by  $A,B,C,O_1,O_2$, as in the figure.
Since at least two edges were removed from the star of $A'$ and only
one edge was added, the degree of  $A'$ strictly decreases  by the operation. If $Z$ is of type $1$, then the edge added
is between $O_1$ and $O_2$, both of which become of degree $4$, and thus the one that is identical to $A'$
 had degree at least $5$ before the operation. This is impossible, since throughout  the
process  degrees of vertices do not increase, and  $\Delta(G) \le 4$.
A similar argument applies if $Z$ is of type $3$, and the added edge is $AO_1$. Thus we may assume that $Z$ is of type $2$,
and that $e=B'O_1'$, see Figure \ref{In the end}.

Since  the degree of $A'$  decreased, and after the addition of $e$ the vertex $O'_1$ has degree $4$, it is impossible that $A'=O'_1$. Hence $A'=B$ and  $B'=O_1$, see Figure \ref{In the end}.
If $Y$ is of Type $1$ or $2$, then
  $A'$ is connected to the two
opponents which are not inside any archipelago. But $A'$ has already
$3$ different neighbors in archipelagos (two from  $Y$ and at least one from $Z$), implying that $A'$
has degree at least $5$ in $G$, a contradiction. Thus
we may assume that $Y$  is
of type $3$.

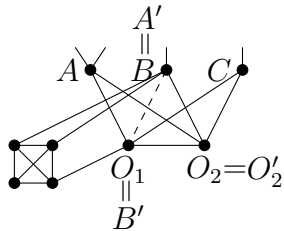
\begin{figure}[htbp] 
\begin{center}
\begin{tikzpicture}
\filldraw[black] (0,0) circle (2pt)node[anchor=east] {$A$}; 
\filldraw[black] (1,0) circle (2pt)node[anchor=east] {$B$};
\filldraw[black] (2,0) circle (2pt)node[anchor=east] {$C$};
\filldraw[black] (0.5,-1) circle (2pt)node[anchor=north] {$O_1$};
\filldraw[black] (1.5,-1) circle (2pt)node[anchor=north] {$O_2$};
\draw (0,0) -- (-0.2,0.3); 
\draw (0,0) -- (0.2,0.3); 
\draw (1,0) -- (1,0.3); 
\draw (2,0) -- (2,0.3); 
\draw (0,0) -- (0.5,-1); 
\draw (0,0) -- (1.5,-1);
\draw[dashed] (1,0) -- (0.5,-1);
\draw (1,0) -- (1.5,-1);
\draw (2,0) -- (0.5,-1);
\draw (2,0) -- (1.5,-1);
\draw (0.5,-1) -- (1.5,-1); 
\filldraw[black] (-0.5,-1) circle (2pt);
\filldraw[black] (-1,-1) circle (2pt);
\filldraw[black] (-0.5,-1.5) circle (2pt);
\filldraw[black] (-1,-1.5) circle (2pt);
\draw (-0.5,-1) -- (-1,-1);
\draw (-0.5,-1) -- (-0.5,-1.5);
\draw (-0.5,-1) -- (-1,-1.5);
\draw (-1,-1) -- (-0.5,-1.5);
\draw (-1,-1) -- (-1,-1.5);
\draw (-0.5,-1.5) -- (-1,-1.5);
\draw (-0.5,-1) -- (1,0);
\draw (-1,-1) -- (1,0);
\draw (-0.5,-1.5) -- (0.5,-1);
\node[](A') at (0.75,0.70){$A'$}; 
\node[rotate=90](=1) at (0.75,0.35){$=$};
\node[](B') at (0.5,-1.95){$B'$};
\node[rotate=90](=2) at (0.5,-1.6){=};
\node[](O'_2) at (2.3,-1.35){$O'_2$};
\node[](=3) at (1.9,-1.35){=};
\end{tikzpicture}
\caption{$C'$ can not be any vertex in this picture, and it must also be connected to $O'_2$, a contradiction.}
\label{In the end}
\end{center}
\end{figure}

Since in a Type $3$ subgraph, $A'$ is connected to $O'_2$, we conclude that $O_2=O'_2$ or else $A'$ would have degree $5$. Since $A',B',C'$ are independent, $C'$ must be a vertex different from $A,B,C,O_1,O_2$, and since in a Type $3$ subgraph $C'$ is connected to $O'_2$ which is equal to $O_2$, this means that  $O_2$ has degree $5$, again a contradiction.
\end{proof}

By the claim it is possible to delete all risky acyclic archipelagos one by one,

We now remove acyclic archipelagos with neighborhoods of size $3$,
adding an edge at each stage,
as follows.
If at the current stage there are no risky acyclic archipelagos we delete any archipelago with independent neighborhood of size $3$ and connect two of its neighbors, without risking the generation of a $K_4$ (since there are no forbidden archipelagos), and without generating a forbidden archipelago (since there are no risky archipelagos). At stages in which there is a risky archipelago
we use the claim to remove such an archipelago, while not generating a risky archipelago, and not generating a $K_4$ (the latter
following from the non-existence of a forbidden archipelago).

Let $G_3$ be the graph obtained after these operations. Then $G_3$ is connected, it does not contain any new $K_4$s, and $\Delta(G_3) \le 4$.


\begin{step}
Removing all remaining archipelagos.

\end{step}

Since $G_3$ does not contain any  acyclic archipelagos with independent neighborhoods of size two or three, by Lemma \ref{happy} we can delete every acyclic archipelago with an independent neighborhood one by one, and after each deletion we can connect some vertices in their neighborhood without creating a new $K_4$. After we deleted every acyclic archipelago with an independent neighborhood, we delete every other  acyclic archipelago and every other cyclic archipelago without adding any additional edges. Let $H$ be the graph obtained from $G_3$ this way.

We claim that $H$ satisfies the requirements of the lemma. Clearly, it is $K_4$-free, and Step \ref{conn} saw to it that it is connected. At each step of our construction degrees of vertices only went down. If in Steps $1$ or $3$ any archipelagos were deleted, the degree of some vertices strictly decreased, since   only one edge is added, while at least $4$ edges were removed as the result of the removal of the archipelago. By Remark \ref{decreasingdegree} each deletion of archipelagos in Step $2$ also decreased the degree of at least one vertex. Condition \ref{strong} in the lemma follows from our construction and Claim \ref{k4cycles}.
\end{proof}

\begin{remark} \label{two-miltonian}
The two-miltonian property of $G$ was used:
\begin{itemize}
\item For the property that $\Delta(G) \leq 4$ .
\item For the property that the $K_4$s in $G$ are vertex disjoint.
\item For the property that an acyclic archipelago sends a certain number of edges to its neighbourhood. This can be avoided since if it would send less, we could treat it as if it is a cyclic archipelago.
\item In Step $1$ to forbid subgraphs like in figure \ref{first case2}. This is important since we have to forbid graphs like \ref{counterexample} as Lemma \ref{technical} can not be applied to such graphs.
\item In Step 2 to guarantee connectedness. This could be avoided by paying attention to the connectedness of the $K_4$-free part at each deletion. (Although this would make the proof even more unpleasant to read.)
\end{itemize}
Thus the authors feel that results similar to Lemma \ref{technical} should hold with assumptions on $G$ that can replace the role of two-miltonicity at the above mentioned parts of the proof.
\end{remark}

\begin{corr} \label{smooth}
In a two-miltonian graph $G$ on $n$ vertices
 $$ \alpha(G) \geq  \zeta(G) + \frac{7}{26}(n-4\zeta(G))  - O(1).$$

\end{corr}
\begin{proof}
Let $H$ be the graph  obtained from $G$ as in  Lemma \ref{technical}. By Theorem \ref{7/26} $\alpha(H) \ge \frac{7}{26}|V(H)|-O(1)=\frac{7}{26}(n-4\zeta(G))-O(1)$, and
 by part (4) in the conclusion of the lemma $\alpha(G) \ge \alpha(H)+\zeta(G)$.
 \end{proof}

\section{Calculating $f(n,n/4)$}\label{sec:nover4}

We will use a theorem that was proved by Albertson, Bollob\'as and Tucker. We state it in a simplified form, tailored to our needs. For the more general version see \cite{abt}

\begin{thm}[M. Albertson, B. Bollobás, S. Tucker] \label{stoneage}
 If $G$ is  $K_4$-free,  $\Delta (G) \leq 4$ and $G$ is not $4$-regular, then $\alpha(G)>\frac{n}{4}$.
\end{thm}

The  value of $f(n,n/4)$ can  be determined for all $n$.

\begin{thm} \label{values} \hfill
\begin{enumerate}
\item
If $4 \nmid n$ then $f(n,n/4)=1$.
\item
$f(4,1)=f(8,2)=3$.
\item
  $f(4k,k)= 2$ for $k\geq3$.
  \end{enumerate}
\end{thm}
\begin{proof}

Part (1)  follows from Brooks' theorem.
The following figure
shows that for every $k$ there exists a two-miltonian graph $G$ with $n=4k$ and $\alpha(G) = \zeta(G)=k$. This means that $f(4k,k) \ge 2$.

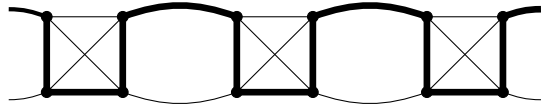
\begin{figure}[htbp,scale=0.5]
\begin{center}
\begin{tikzpicture}
\filldraw[black] (0,0) circle (2pt);
\filldraw[black] (1,0) circle (2pt);
\filldraw[black] (0,1) circle (2pt);
\filldraw[black] (1,1) circle (2pt);
\filldraw[black] (2.5,0) circle (2pt);
\filldraw[black] (3.5,0) circle (2pt);
\filldraw[black] (2.5,1) circle (2pt);
\filldraw[black] (3.5,1) circle (2pt);
\filldraw[black] (5,0) circle (2pt);
\filldraw[black] (6,0) circle (2pt);
\filldraw[black] (5,1) circle (2pt);
\filldraw[black] (6,1) circle (2pt);
\draw[thin] (-0.5,-0.1) to[out=0,in=200] (0,0);
\draw[ultra thick] (-0.5,1.1) to[out=0,in=160] (0,1);
\draw[line width =2.5pt] (0,0) -- (0,1);
\draw[line width =2.5pt] (0,0) -- (1,0);
\draw[thin] (0,0) -- (1,1);
\draw[thin] (0,1) -- (1,0);
\draw[thin] (0,1) -- (1,1);
\draw[line width =2.5pt] (1,0) -- (1,1);
\draw[thin] (1,0) to[out=-20,in=200] (2.5,0);
\draw[line width =2.5pt] (1,1) to[out=20,in=160] (2.5,1);
\draw[line width =2.5pt] (2.5,0) -- (3.5,0);
\draw[line width =2.5pt] (2.5,0) -- (2.5,1);
\draw[thin] (2.5,0) -- (3.5,1);
\draw[thin] (3.5,0) -- (2.5,1);
\draw[line width =2.5pt] (3.5,0) -- (3.5,1);
\draw[thin] (2.5,1) -- (3.5,1);
\draw[thin] (3.5,0) to[out=-20,in=200] (5,0);
\draw[line width =2.5pt] (3.5,1) to[out=20,in=160] (5,1);
\draw[line width =2.5pt] (5,0) -- (6,0);
\draw[line width =2.5pt] (5,0) -- (5,1);
\draw[thin] (5,0) -- (6,1);
\draw[thin] (6,0) -- (5,1);
\draw[line width =2.5pt] (6,0) -- (6,1);
\draw[thin] (5,1) -- (6,1);
\draw[thin] (6,0) to[out=-20,in=180] (6.5,-0.1);
\draw[line width =2.5pt] (6,1) to[out=20,in=180] (6.5,1.1);
\end{tikzpicture}
\caption{The two strips eventually close on themselves.}
\label{covered}
\end{center}
\end{figure}

 Since there are only three distinct Hamiltonian cycles on $4$ vertices, the fact that $f(4,1)= 3$ is easy. Figure~\ref{3example} is an example of three Hamiltonian cycles on the same vertex set of size $8$, having each pairwise union $K_4$ -covered, thus $\alpha(C_i \cup C_j)=2$ for $1 \le i <j \le 3$,  showing that $f(8,2)\ge 3$.

 \begin{figure}[htbp]
 \begin{tikzpicture}
\filldraw[black] (1,1) circle (2pt);
\filldraw[black] (1,2) circle (2pt);
\filldraw[black] (2,1) circle (2pt);
\filldraw[black] (2,2) circle (2pt);
\filldraw[black] (0,0) circle (2pt);
\filldraw[black] (3,0) circle (2pt);
\filldraw[black] (0,3) circle (2pt);
\filldraw[black] (3,3) circle (2pt);
\draw (-0.5,-0.5) -- (0,0);
\draw (0,0) -- (0,3);
\draw (0,3) -- (1,1);
\draw (1,1) -- (1,2);
\draw (1,2) -- (2,1);
\draw (2,1) -- (2,2);
\draw (2,2) -- (3,0);
\draw (3,0) -- (3,3);
\draw (3,3) -- (3.5,3.5);
\end{tikzpicture}
\begin{tikzpicture}
\filldraw[black] (1,1) circle (2pt);
\filldraw[black] (1,2) circle (2pt);
\filldraw[black] (2,1) circle (2pt);
\filldraw[black] (2,2) circle (2pt);
\filldraw[black] (0,0) circle (2pt);
\filldraw[black] (3,0) circle (2pt);
\filldraw[black] (0,3) circle (2pt);
\filldraw[black] (3,3) circle (2pt);
\draw (0,0) -- (1,1);
\draw (1,1) -- (3,0);
\draw (3,0) -- (2,1);
\draw (2,1) -- (3,3);
\draw (3,3) -- (2,2);
\draw (2,2) -- (0,3);
\draw (0,3) -- (1,2);
\draw (1,2) -- (0,0);
\node[]() at (1,-0.3){};
\end{tikzpicture}
\begin{tikzpicture}
\filldraw[black] (1,1) circle (2pt);
\filldraw[black] (1,2) circle (2pt);
\filldraw[black] (2,1) circle (2pt);
\filldraw[black] (2,2) circle (2pt);
\filldraw[black] (0,0) circle (2pt);
\filldraw[black] (3,0) circle (2pt);
\filldraw[black] (0,3) circle (2pt);
\filldraw[black] (3,3) circle (2pt);
\draw (0,0) -- (2,1);
\draw (2,1) -- (1,1);
\draw (1,1) -- (2,2);
\draw (2,2) -- (1,2);
\draw (1,2) -- (3,3);
\draw (3,3) -- (0,3);
\draw (0,3) -- (-0.5,3.5);
\draw (3.5,-0.5) -- (3,0);
\draw (3,0) -- (0,0);
\end{tikzpicture}
\caption{The vertices are identified by horizontal shifting, showing that $f(8,2)\ge 3$}
\label{3example}
\end{figure}
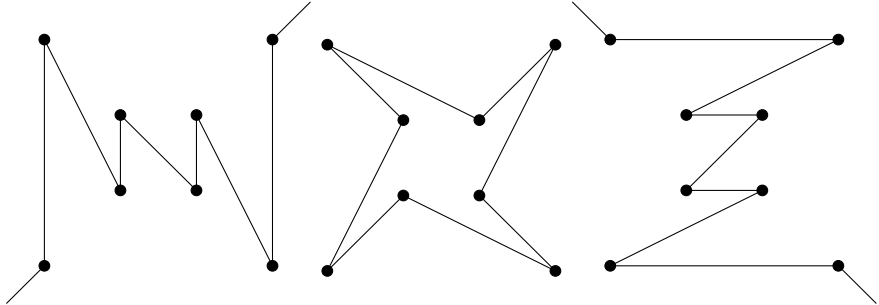

Each pair  of dangling edges in the two extreme cycles are meant to join to form one edge. For our next arguments we will need the  following lemma.

\begin{lem} \label{nothree}
If $4 \mid n$ and $n \geq 12$ then there do not exist three Hamiltonian cycles  $C_1,C_2,C_3$ such that  $C_i \cup C_j$ is
 $K_4$-covered for all pairs $1 \le i <j \le 3$.
 \end{lem}
\begin{proof}
Assume for contradiction that we do have three such cycles. Enumerate the vertices $1, \ldots ,12$ so that
$i(i+1)\in E(C_1)$ for all $i \le 12$ (cyclical counting), and $1,2,3,4$ form a $K_4$ in $C_1 \cup C_2$.
  Then the  edges of $C_2$ inside the three  $K_4$s forming $C_1 \cup C_2$ must be as in Figure~\ref{covered1}.


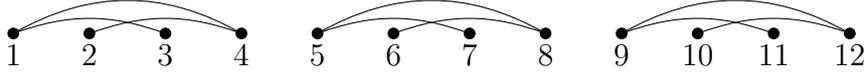
\begin{figure}[htbp] 
\begin{center}
\begin{tikzpicture}
\filldraw[black] (1,0) circle (2pt)node[anchor=north] {$1$};
\filldraw[black] (2,0) circle (2pt)node[anchor=north] {$2$};
\filldraw[black] (3,0) circle (2pt)node[anchor=north] {$3$};
\filldraw[black] (4,0) circle (2pt)node[anchor=north] {$4$};
\filldraw[black] (5,0) circle (2pt)node[anchor=north] {$5$};
\filldraw[black] (6,0) circle (2pt)node[anchor=north] {$6$};
\filldraw[black] (7,0) circle (2pt)node[anchor=north] {$7$};
\filldraw[black] (8,0) circle (2pt)node[anchor=north] {$8$};
\filldraw[black] (9,0) circle (2pt)node[anchor=north] {$9$};
\filldraw[black] (10,0) circle (2pt)node[anchor=north] {$10$};
\filldraw[black] (11,0) circle (2pt)node[anchor=north] {$11$};
\filldraw[black] (12,0) circle (2pt)node[anchor=north] {$12$};
\draw (1,0) to[out=20,in=160] (3,0); 
\draw (1,0) to[out=30,in=150] (4,0);
\draw (2,0) to[out=20,in=160] (4,0);
\draw (5,0) to[out=20,in=160] (7,0);
\draw (5,0) to[out=30,in=150] (8,0);
\draw (6,0) to[out=20,in=160] (8,0);
\draw (9,0) to[out=20,in=160] (11,0);
\draw (9,0) to[out=30,in=150] (12,0);
\draw (10,0) to[out=20,in=160] (12,0);
\end{tikzpicture}
\caption{The edges of $C_2$ that form the $K_4$s in $C_1 \cup C_2$ on the twelve vertices that we focus on.}
\label{covered1}
\end{center}
\end{figure}

Degree considerations and the fact that $C_2$ is Hamiltonian yield that it is necessarily edge disjoint from $C_1$. In general,
all three cycles are  edge disjoint.

Since $C_1 \cup C_3$ is also $K_4$-covered, we can also draw the edges of $C_3$ that form the $K_4$s in $C_1 \cup C_3$ on the same vertex set: $\{1,\ldots, 12\}$. This must be very similar to Figure~\ref{covered1}, but it might be shifted as we cannot assume that the vertices $1,2,3,4$ form a $K_4$ in $C_1 \cup C_3$. Moreover, since we already know that the cycles are edge disjoint, it should be shifted by exactly two vertices. Figure~\ref{covered2} describes the union of the edges of $C_2$ and $C_3$ that form the $K_4$s in their union with $C_1$. This is also a subgraph of $C_2 \cup C_3$.


\begin{figure}[htbp] 
\begin{center}
\begin{tikzpicture}
\filldraw[black] (1,0) circle (2pt)node[anchor=south] {$1$};
\filldraw[black] (2,0) circle (2pt)node[anchor=north] {$2$};
\filldraw[black] (3,0) circle (2pt)node[anchor=north] {$3$};
\filldraw[black] (4,0) circle (2pt)node[anchor=south] {$4$};
\filldraw[black] (5,0) circle (2pt)node[anchor=south] {$5$};
\filldraw[black] (6,0) circle (2pt)node[anchor=north] {$6$};
\filldraw[black] (7,0) circle (2pt)node[anchor=north] {$7$};
\filldraw[black] (8,0) circle (2pt)node[anchor=south] {$8$};
\filldraw[black] (9,0) circle (2pt)node[anchor=south] {$9$};
\filldraw[black] (10,0) circle (2pt)node[anchor=north] {$10$};
\filldraw[black] (11,0) circle (2pt)node[anchor=north] {$11$};
\filldraw[black] (12,0) circle (2pt)node[anchor=south] {$12$};
\draw (1,0) to[out=20,in=160] (3,0); 
\draw (1,0) to[out=30,in=150] (4,0);
\draw (2,0) to[out=20,in=160] (4,0);
\draw (5,0) to[out=20,in=160] (7,0);
\draw (5,0) to[out=30,in=150] (8,0);
\draw (6,0) to[out=20,in=160] (8,0);
\draw (9,0) to[out=20,in=160] (11,0);
\draw (9,0) to[out=30,in=150] (12,0);
\draw (10,0) to[out=20,in=160] (12,0);
\draw (3,0) to[out=-20,in=-160] (5,0); 
\draw (3,0) to[out=-30,in=-150] (6,0);
\draw (4,0) to[out=-20,in=-160] (6,0);
\draw (7,0) to[out=-20,in=-160] (9,0);
\draw (7,0) to[out=-30,in=-150] (10,0);
\draw (8,0) to[out=-20,in=-160] (10,0);

\filldraw[black] (12.15,-0.3) circle (0.8pt);
\filldraw[black] (12,-0.3) circle (0.8pt);
\filldraw[black] (11.85,-0.3) circle (0.8pt);

\filldraw[black] (1.15,-0.3) circle (0.8pt);
\filldraw[black] (1,-0.3) circle (0.8pt);
\filldraw[black] (0.85,-0.3) circle (0.8pt);

\end{tikzpicture}
\caption{A subgraph of $C_2 \cup C_3$.}
\label{covered2}
\end{center}
\end{figure}
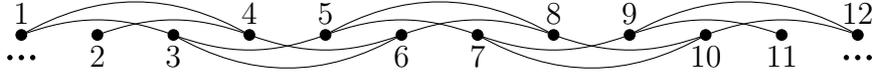

The edges in Figure \ref{covered2} form a $3$-regular subgraph of $C_2 \cup C_3$, thus every vertex has an additional neighbor (since $C_2$ is edge disjoint from $C_3$). Observe that if the vertices labeled $5,6,9$ form an independent set, their neighborhood is of size at most $9$ in $C_2 \cup C_3$ and by Observation \ref{greedy} we can enlarge it to an independent set of size more than $n/4$ in $C_2 \cup C_3$, a contradiction. Thus there must be an edge connecting some of the vertices labeled $5,6,9$. In Figure \ref{covered2} the vertices $5,9$ already have two edges from $C_2$, and the vertex $6$ has two edges from $C_3$, so no edge can connect $6$ to $5$ or $9$. Thus $5$ and $9$ must be connected (by an edge in $C_3$). But then by shifting the whole argument to the left by four, we get that $1$ should be also connected to $5$ in $C_3$, which contradicts the fact that $C_3$ is a cycle.
\end{proof}

The proof that $ f(8,2)\le 3$ and $f(12,3)\leq 2$ can be done by computer.
It  remains to be shown that   $f(n,n/4) \le 2$ when $n$ is divisible by four and $n \geq 16 $. Assume for contradiction that $f(n,n/4) \geq 3$ thus there exist three Hamiltonian cycles   $C_1,C_2,C_3$  on $n=4k$ vertices, such that
 $\alpha(C_i \cup C_j)=\frac{n}{4}$ whenever $1 \le i<j \le 3$.
Then each union  $C_i \cup C_j$ contains a copy of $K_4$, since otherwise
 by Theorem \ref{7/26} there exists  an independent set of size $(7n-4)/26$, which is strictly larger than $n/4$ when $n \geq 16$. We next show that  $C_i \cup C_j$ not only contains a single $K_4$, but it is $K_4$-covered.
Assuming that this is not the case, since there is at least one $K_4$ in $G$, by Lemma \ref{technical} there exists a $K_4$-free nonempty subgraph $H$ that has at least one vertex of degree at most three. But then Theorem \ref{stoneage} yields an independent set in $H$ strictly larger than $|V(H)|/4$, and by Lemma \ref{technical} we can enlarge it to an independent set of size more than $n/4$ in $C_i \cup C_j$. Thus $C_i \cup C_j$ must be $K_4$-covered, but this contradicts Lemma~\ref{nothree} and the proof is complete.
\end{proof}

\begin{remark}
Supopse that we are interested in the maximal number of Hamiltonian paths (instead of Hamiltonian cycles) with the property that the union of any two has independence number at most $n/4$. It can be proven that when $n$ is divisible by four we can have at most two Hamiltonian paths (and we can have two, see Figure \ref{covered} without the strips closing on themselves) and otherwise we can only have a single one by the usual Brooks reasoning. In this context, $n=8,12$ are exceptional only because we are interested in Hamiltonian cycles instead of paths.
 \end{remark}

\section{A lower bound on $c_t$}\label{sec:lowerbounds}

We will use the following observation.

\begin{claim}\label{csoka} \cite{csoka}
Let $y$ be a vertex of $G$ with exactly two neighbors $x$ and $z$ such that $x$ and $z$ are not connected. Let $G'$ be a graph defined by  $V(G')=V(G)\setminus \{x,y,z\} \cup \{v\}$ where $v$  is a new vertex connected to all remaining neighbors of $x$ and $z$. Thus $\text{deg}(v)=\text{deg}(x)+\text{deg}(z)-2$. Then for any independent set $I'$ of $G'$, we can construct an independent set $I$ of $G$ of such that $|I|=|I'|+1$.
\end{claim}
\begin{proof}
If $v \in I'$ then $I=I' \setminus \{v\} \cup \{x,z\} $. If $v \notin I'$ then $I = I' \cup \{y\}$.
\end{proof}

\begin{notation}
 We call a subgraph of
 $G$  {\em good } if it is an induced path of length three
 We write $\psi(G)$ for the maximal number of vertex disjoint good subgraphs in $G$.
\end{notation}

In a two-miltonian graph all copies of $K_4$ are vertex disjoint, so
$\zeta$ is the number of disjoint copies of $K_4$.

\begin{observation}\label{psizeta}
Let $C, D_1, D_2$  be  Hamiltonian cycles, and let $m$ be the number of $K_4$s that are contained in both $C \cup D_1$ and $C \cup D_2$. Then
$\psi(D_1 \cup D_2) \ge m$.
\end{observation}

\begin{proof}
In every $K_4$ contained in both $C \cup D_1$ and $C \cup D_2$ the edges  that do not belong to $C$ form a path of length $3$  in  both $D_1$ and $D_2$.
\end{proof}

\begin{lem} \label{quality_indep}
If $G$ is $2$-miltonian then

$$\alpha(G) \geq   \frac{7}{26}n- \frac{1}{13}\zeta+\frac{1}{2}\psi -O(1)$$

\end{lem}
\begin{proof}

Let  $e=|E(G)|$.
By Theorem \ref{notevenmyfinalform}
$e-9n+26\alpha(G) \geq -4$.

Let $G_1 := C_1 \cup C_2$. Let $T$ be a set of disjoint good subgraphs
 of  size $\psi$.
  For each path  $P_i\in T$,  using the notation of  Figure \ref{finish him}  below, we apply the operation described in
 Claim \ref{csoka}, of removing $x,y,z$ and adding a vertex $v=v_i$ connected to the remaining neighbors of $x,z$. Let $G_2$ be the graph obtained by combining all these $\psi$ operations. Then $|V(G_2)| =n-2\psi$.

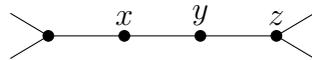
\begin{figure}[htbp,scale=0.5]
\begin{center}
\begin{tikzpicture}
\filldraw[black] (0,0) circle (2pt);
\filldraw[black] (1,0) circle (2pt)node[anchor=south] {$x$};
\filldraw[black] (2,0) circle (2pt)node[anchor=south] {$y$};
\filldraw[black] (3,0) circle (2pt)node[anchor=south] {$z$};
\draw (-0.5,0.3) -- (0,0);
\draw (-0.5,-0.3) -- (0,0);
\draw (0,0) -- (1,0);
\draw (1,0) -- (2,0);
\draw (2,0) -- (3,0);
\draw (3,0) -- (3.5,0.3);
\draw (3,0) -- (3.5,-0.3);
\end{tikzpicture}
\caption{A good subgraph in $C_1 \cup C_2.$}
\label{finish him}
\end{center}
\end{figure}

 Observe that $G_2$ is still two-miltonian and that $\zeta(G_2)=\zeta(G)$.
  Let $H$ be the graph obtained from $G_2$ using  Lemma \ref{technical}. Then  $H$ is a subgraph of $G_2$,
   it is simple, connected, and $K_4$-free. We have $|V(H)|=n-4\zeta-2\psi$, and since every vertex $v_i$, and its unnamed neighbor in Figure \ref{finish him} has degree at most $3$  in $H$, we have
   $$2|E(H)|\leq 4((n-4\zeta-2\psi)-2\psi) +6\psi=4n-16\zeta -10\psi.$$

    Thus using the inequality in the second remark of Theorem \ref{notevenmyfinalform} we get that

$$\alpha(H) \geq  \frac{9|V(H)|-|E(H)|-4}{26} \geq  \frac{9(n-4\zeta-2\psi)-(2n-8\zeta-5\phi)}{26} -\frac{4}{26} =$$
$$= \frac{7}{26}n-\frac{28}{26}\zeta-\frac{13}{26}\psi -O(1)=\frac{7}{26}n-\frac{14}{13}\zeta-\frac{1}{2}\psi-O(1).$$

By Lemma \ref{technical} we can enlarge this independent set to an independent set of $G_2$ of size

$$\alpha(G_2)\geq \frac{7}{26}n- \frac{1}{13}\zeta-\frac{1}{2}\psi -O(1). $$

By Claim \ref{csoka} we have an independent set in $G_1$ of size

$$ \alpha(G_1) \geq \frac{7}{26}n- \frac{1}{13}\zeta+\frac{1}{2}\psi -O(1) $$

finishing the proof.
\end{proof}


\begin{lem} \label{johnson}
Let $\varepsilon>0$ be fixed and $S=\{S_1, \ldots, S_m\}$ be a set system on a ground set of size $n$ with the following properties.

\begin{itemize}
\item $\forall i: $ $|S_i| \geq xn$
\item $\forall i\neq j: $ $(1-\varepsilon)x^2n \geq|S_i \cap S_j|$
\end{itemize}
Then $m$ is bounded by a number independent of $n$: $$m \leq
q(x,\varepsilon)= \frac{1-x(1-\varepsilon)}{x\varepsilon}.$$
\end{lem}
\begin{proof}

Let $z=xm$. Let $X_1$ and $X_2$ be two uniformly randomly and
independently chosen sets from $S$. Let us denote by $l_i$ the
number of sets in $S$ which contain the element $i$. Now we have
that:

$$ \mathbb{E}(|X_1 \cap X_2|)= \sum_{i=1}^n\mathbb{P}(i \in X_1 \cap X_2)=\sum_{i=1}^n \frac{\binom{l_i}{2}}{\binom{m}{2}} $$

Since the average of the $l_i$ is exactly $z$, and the function $\frac{x(x-1)}{2}$ is convex, by Jensen's inequality we have the following

$$ \sum_{i=1}^n \frac{\binom{l_i}{2}}{\binom{m}{2}} \geq n \frac{\binom{z}{2}}{\binom{m}{2}}=n\frac{z(z-1)}{\frac{z}{x}\left(\frac{z}{x}-1\right)} =n \frac{z-1}{\frac{z}{x^2}-\frac{1}{x}}=nx^2 \frac{z-1}{z-x}. $$

Elementary calculation yields that the inequality

$$\frac{z-1}{z-x}\leq (1-\varepsilon)$$

holds if and only if $ m=\frac{z}{x} \leq \frac{1-x(1-\varepsilon)}{x\varepsilon}$, finishing the proof.

\end{proof}

\begin{remark}
If in Lemma \ref{johnson} we  replace $(1-\varepsilon)x^2n$ by
$(1+\varepsilon)x^2n$, we can construct set systems of exponential
size by a uniform random construction.
\end{remark}

For $0<x \le 1$ and $\varepsilon >0$  let $$\delta(x,\varepsilon)= \left(q\left(\frac{x}{4},\varepsilon \right)+1\right)^{-1} =\left(\frac{4-x(1-\varepsilon)}{x\varepsilon}+1\right)^{-1}.$$

\begin{lem} \label{iterating} For every $0<x \le 1$ and $\varepsilon >0$
there exists a number $\theta(x,\varepsilon)$ such that
if $k > \theta(x,\varepsilon)$ and
 $X=\{C_1 , \ldots , C_{k}\}$ is a collection of Hamiltonian cycles satisfying
 $\frac{xn}{4} \leq \zeta(C_i \cup C_j)$
  whenever $1 \leq i < j \leq k$
 then there exists a subcollection $X_2 \subseteq X$ of size at least $|X|^{\delta(x,\varepsilon)}$, such that $\frac{(1-\varepsilon)x^2n}{16}\leq \psi(C_k \cup C_l)$  for every pair of cycles $C_k,C_l \in X_2$.

\end{lem}
\begin{proof}
Let $A$ be a graph whose vertex set is $X$, and two cycles $C_i$ and $C_j$ are connected by an edge if and only if $\frac{(1-\varepsilon)x^2n}{16} \leq \psi(C_i \cup C_j)$. Our aim is to show that $|X|^{\delta(x,\varepsilon)} \leq \omega(A)$ (the latter denoting the largest size of a clique in $A$). This will follow from Ramsey's theorem and an upper bound we shall obtain on $\alpha(A)$.
\begin{claim}\label{smallalpha}
 $\alpha(A) \le q(\frac{x}{4},\varepsilon)+1=\frac{4-x(1-\varepsilon)}{x\varepsilon}+1$.
\end{claim}
\begin{proof}
Suppose to the contrary that there exists an independent set $D_1,D_2, \ldots D_p$ in $A$, where $p=\lceil q(\frac{x}{4},\varepsilon)\rceil+2$.  By relabeling the vertices, we can assume that $D_1$ is the cycle $(1,2,\ldots,n)$. For every $1< j \le p$ let $S_j$ be the set of those $1 \le i \le n$ for which $D_1 \cup D_j$ contains a
$K_4$ on the vertices $i, i+1, i+2, i+3 (\bmod n)$.

By the assumption of the lemma   $\zeta(D_1 \cup D_j) \ge \frac{xn}{4}$, and hence  $|S_j| \ge \frac{xn}{4} $ for all $j \le p$.
Since the cycles $D_j$ are independent in $A$,  $\psi(D_i\cup D_j) < \frac{(1-\varepsilon)x^2}{16}$  whenever   $i \neq j$. By Observation \ref{psizeta} this implies that $|S_i \cap S_j| < \frac{(1-\varepsilon)x^2}{16}n$. \\
These combined yield a  contradiction to Lemma \ref{johnson}.
\end{proof}

By a result of Ajtai, Komlos and Szemeredi \cite{aks}, for fixed $s$  and $t$ large enough  we have the following bound on the Ramsey numbers: $$R(s,t) \leq c_s \frac{t^{s-1}}{\log(t)^{s-2}}, $$

implying $R(s,t) \leq  t^s$  for fixed $s$ and large enough $t$. Thus for fixed $s$ and large enough $n$ if G is a graph on $n$ vertices with $\alpha(G)\leq s$ then $\omega(G)\geq n^{\frac{1}{s}}$. Applying this to the graph $A$, and using Claim \ref{smallalpha}, we obtain that $\omega(A) \geq |X|^{\left(q(\frac{x}{4},\varepsilon)+1\right)^{-1}}=|X|^{\delta(x, \varepsilon)}$ and the proof is complete.
\end{proof}

For a family $X$ of Hamiltonian cycles let $m(X)=\min_{C \neq D \in X}\frac{\zeta(C \cup D)}{n} \in [0,1]$. For the sake of readability, we will often write $m$ for $m(X)$.

\begin{lem} \label{step}
Let $X$ be a set of Hamiltonian cycles, if
$\frac{\psi(C\cup D)}{n} < (1-\varepsilon)\left(
\frac{\zeta(C \cup D)}{n}\right)^2-\varepsilon$ for every pair $C\neq D$ of cycles in $X$ then there exists a subset $Y$ of $X$ of size at least $|X|^{\delta(4m, \varepsilon)}$
such that
 $$ m(Y)^2 > m(X)^2 +\frac{ \varepsilon}{(1-\varepsilon)}. $$
\end{lem}

 \begin{proof}
By Lemma
\ref{quality_indep} and the fact that $d$ is positive $m>0$.
Applying Lemma \ref{iterating} with $x=4m$, we obtain $Y \subseteq X$  of size at least $ |X|^{\delta(4m,\varepsilon )}$,
such that every pair of cycles $C,D \in Y$ satisfies $\frac{\psi(C \cup D)}{n} \geq (1-\varepsilon)m^2$.
Let $C,D \in Y$ be  such that $\frac{\zeta(C\cup D)}{n}=m(Y)$.  By the  assumption of the lemma

 $$ (1-\varepsilon)m(Y)^2-\varepsilon = (1-\varepsilon)\left( \frac{\zeta(C\cup D)}{n}\right)^2-\varepsilon  > \frac{\psi(C \cup D)}{n} \geq  (1-\varepsilon)m^2, $$
which yields the desired result.

\end{proof}
\begin{corr}\label{exists}
If $X$ is a set of Hamiltonian cycles satisfying
$$|X|^{\left( \delta(4m,\varepsilon)^{\frac{1-\varepsilon}{\varepsilon}}\right)} > \theta(m, \varepsilon)$$
then there exists a pair $C\neq D$ of cycles in $X$ such that
$$\frac{\psi(C\cup D)}{n} \geq (1-\varepsilon)\left(
\frac{\zeta(C \cup D)}{n}\right)^2-\varepsilon.$$
\end{corr}
\begin{proof}
Assume negation. Applying Lemma \ref{step} repeatedly, we obtain then
 a sequence $X=X_1 \supseteq X_2 \supseteq \ldots \supseteq X_p$ of sets of Hamiltonian cycles, such that $m(X_{i+1})^2 \ge m(X_{i})^2 +
 \frac{\varepsilon}{(1-\varepsilon)}$ and $|X_{i+1}| \ge |X|^{\delta(4m(X_i), \varepsilon)}$  for all $i<p$.
 Since $\delta(x,\varepsilon)$ is increasing in $x$ and the sequence $m(X_i)$ is increasing, the assumption on the size of $X$ thus leads to the conclusion that for $p$ as large as  $\frac{1-\varepsilon}{\varepsilon}+1$
  we still have $X_p \neq \emptyset$. But this yields $m(X_p) >1$, which is impossible since by definition $m \in [0,1]$.
 \end{proof}

We can now obtain our goal - a lower bound on the threshold constant $c_t$. Remember that $c_t$ is a real number such that for $c<c_t$ the value of $f(n,cn)$ is bounded by a constant independently of $n$, and for $c >c_n$ this value is exponential in $n$.

\begin{thm}
 $c_t\ge\frac{45}{169} \approx 0.26627.$
\end{thm}
\begin{proof}

Let $\varepsilon>0$ be fixed, and let $d$ be positive such that
$c_t<\frac{7}{26}-d$. Let $n$ be large and $X$ be a large collection of Hamiltonian cycles on $n$ vertices such that for every pair  $C\neq D$ of cycles in $X$ we have $\alpha(C \cup D)\leq (\frac{7}{26}-d)n$. Here ``large'' is dictated by Corollary \ref{exists}

By Corollary \ref{exists} there exist cycles $C\neq D \in X$ for which $\frac{\psi(C\cup D)}{n} \geq (1-\varepsilon)\left( \frac{\zeta(C\cup D)}{n} \right)^2-\varepsilon$. By Lemma \ref{quality_indep}

$$\left(\frac{7}{26}-d\right)n \geq \alpha(C \cup D) \geq  \frac{7}{26}n- \frac{1}{13}\zeta(C \cup D)+\frac{1}{2}\psi(C \cup D) -O(1) $$

and thus

$$-d \geq -\frac{1}{13}\frac{\zeta(C \cup D)}{n}+\frac{1}{2}\frac{\psi(C \cup D)}{n} -\frac{O(1)}{n}$$

$$-d \geq -\frac{1}{13}\frac{\zeta(C \cup D)}{n}+\frac{1}{2}((1+\varepsilon)\left(\frac{\zeta(C\cup D)}{n}\right)^2-\varepsilon) -O(1)$$

Since we can choose $n$ arbitrarily large and $\varepsilon$ arbitrarily small, it follows that

 $$-d \geq -\frac{1}{13}\frac{\zeta(C \cup D)}{n}+\frac{1}{2}\left(\frac{\zeta(C\cup D)}{n}\right)^2$$

by taking the minimum of the right hand side we get that

$$d \leq \frac{1}{338} \approx 0.002958 $$

for every choice of $d$ where $\frac{7}{26}-d >c_t$, thus $c_t \geq \frac{7}{26}-\frac{1}{338}=\frac{45}{169} \approx 0.266272$

\end{proof}


\end{document}